\documentclass{amsart}
\usepackage[english]{babel}

\usepackage[utf8]{inputenc}

\usepackage{verbatim}

\usepackage{amsmath, amssymb,latexsym,amsthm}

\usepackage{hyperref}
\usepackage{enumitem}
\usepackage{moreenum}

\theoremstyle{definition}

\numberwithin{example}{section}

\usepackage[utf8]{inputenc}
\usepackage{graphicx}
\graphicspath{ {images/} }
\usepackage{caption}
\usepackage{subcaption}
\usepackage{float}
\usepackage{fancyhdr}
\renewcommand{\tilde}{\widetilde}

\newcommand{\elemsub}{\preccurlyeq}

\renewcommand{\hat}{\widehat}
\usepackage{amsmath}
\usepackage{amsfonts}
\usepackage{amssymb}
\usepackage{amsthm}
\usepackage{mathtools}
\usepackage{tikz}
\usetikzlibrary{positioning,shapes.geometric,fit, calc}

\usepackage{cancel}
\usepackage{enumitem} 

 \AtBeginDocument{\def\MR#1{}}

\newtheoremstyle{normal}
{}
{}
{\normalfont}
{}
{\bfseries}
{}
{0.8em}
{}
 
\newtheoremstyle{kursiv}
{}
{}
{\itshape}
{}
{\bfseries}
{}
{0.8em}
{}
{}

\theoremstyle{kursiv}
\newtheorem{thm}{Theorem}[section]
\newtheorem{lem}[thm]{Lemma}		 				 
\newtheorem*{claim*}{Claim}		
\newtheorem{cor}[thm]{Corollary}		
\newtheorem{prop}[thm]{Proposition}

\theoremstyle{normal} 
\newtheorem{df}[thm]{Definition}

\newtheorem{remark}[thm]{Remark}

\newcommand{\calK}{\mathcal{K}}

\newcommand{\Cl}{\mathrm{Cl}}
\newcommand{\G}{\Gamma}

\DeclareMathOperator{\acl}{acl}
\DeclareMathOperator{\Aut}{Aut}     

\DeclareMathOperator{\dcl}{dcl}

\DeclareMathOperator{\tp}{tp}

\newcommand{\M}{\mathfrak{M}}
\newcommand{\g}{\gamma}

\renewcommand{\phi}{\varphi}

\def\Ind#1#2{#1\setbox0=\hbox{$#1x$}\kern\wd0\hbox to 0pt{\hss$#1\mid$\hss}
\lower.9\ht0\hbox to 0pt{\hss$#1\smile$\hss}\kern\wd0}

\def\ind{\mathop{\mathpalette\Ind{}}}

\def\notind#1#2{#1\setbox0=\hbox{$#1x$}\kern\wd0
\hbox to 0pt{\mathchardef\nn=12854\hss$#1\nn$\kern1.4\wd0\hss}
\hbox to 0pt{\hss$#1\mid$\hss}\lower.9\ht0 \hbox to 0pt{\hss$#1\smile$\hss}\kern\wd0}

\def\inds{\left.\ind\!\!\!\right.^*}

\begin{document}
\title{On the model theory of open generalized polygons}
\date{\today}
\author{Anna-Maria Ammer}
\thanks{AM: ammer@gmail.com}
\author{Katrin Tent}
\thanks{KT: Universit\"at M\"unster,  tent@wwu.de}
\thanks{The authors were partially supported by the Deutsche Forschungsgemeinschaft (DFG, German Research Foundation) under Germany's Excellence Strategy EXC 2044-390685587, Mathematics M\"unster: Dynamics-Geometry-Structure and SFB 1442.}


\maketitle

\begin{abstract}

We show that for any $n\geq 3$ the theory of open generalized $n$-gons is complete, decidable and strictly stable,  yielding a new class of examples in the zoo of stable theories.

\end{abstract}

\section{Introduction}
\label{intro_section}

Generalized polygons  were introduced by Tits in order to give geometric interpretations of the groups of Lie type in rank $2$, in the same way that projective planes  correspond to groups of type $A_2$. In fact, generalized polygons are the rank $2$ case of spherical buildings. A generalized $n$-gon is a bi-partite graph with diameter $n$ (i.e. any two vertices have distance at most $n$), girth $2n$ (i.e. the smallest cycles have length $2n$) and such that all vertices have valency at least $3$. Clearly, for $n=2$ such a graph is simply a complete bi-partite graph and in what follows we will always assume $n\geq 3$. Thinking of the bi-partition as corresponding to points and lines, we see that the case $n=3$ is simply a different way of phrasing the axioms of a projective plane, namely: any two points lie on a unique line, any two lines intersect in a unique point and every line contains at least three points. (It then easily follows that every point has at least three lines passing through it.) Remarkably, if  the graph is finite, then by a fundamental result of Feit and Higman \cite{FH} the only possible values for $n$ are $3, 4, 6,$ and $8$. Similar restrictions hold for other well-behaved or \emph{tame} categories of generalized polygons, e.g. if one  assumes that the underlying sets of vertices are compact, or algebraic, one obtains the same restrictions. Since we tend to think of finite Morley rank as a rather strong tameness assumption it might be remarkable that  this restriction does not hold in this setting, see \cite{veryhomogeneous}.

In fact, it is easy to see that infinite generalized $n$-gons exist for any $n\geq 3$: starting with a finite bi-partite configuration that does not contain any $2k$-cycles for $k<n$, one can easily complete this by freely adding enough paths in order to make sure that the graph has diameter $n$ (see Definition~\ref{def: free completion} below). In fact,  such constructions yield the only known examples of generalized $n$-gons for $n\neq 3,4,6, 8$.

Free projective planes were already studied by M. Hall \cite{Hall}, Siebenmann~\cite{Siebenmann},  and Kopejkina~\cite{Kopeikina} and their model theory was studied in \cite{forest}, \cite{HP} and \cite{TZi}. The theory of the free projective planes is strictly stable by \cite{HP} and the notion of independence in the sense of stability agrees with the one studied in \cite{pseudospace} and \cite{2-ample}. In this note we extend the results from \cite{forest} and \cite{HP} to open generalized polygons, using the methods developed in \cite{veryhomogeneous}, \cite{pseudospace} and \cite{2-ample}. In particular, it was shown in \cite{HP} that the theory of open projective planes is complete, strictly stable, does not have a prime model and has uncountably many non-isomorphic countable models.

\section{Generalized polygons}

We first recall  some graph-theoretic notions.
For $a$ and $b$ in $A$, the \emph{distance} $d(a,b)$ between $a$ and $b$ is the smallest number $m$ for which there is a path $a=a_0,a_1,...,a_m=b$ with $a_i$ in $A$, where $a_i$ and $a_{i+1}$ are incident for $0\leq i<m.$ We may write $d_A(a,b)$ to emphasize the dependence on the graph $A$.\\

The \emph{girth} of a graph $A$ is the length of a shortest cycle in $A$.
The \emph{diameter} of a graph $A$ is the maximal distance between two elements in $A$. 
We say that a subgraph $A$ of a graph $B$ is \emph{isometrically embedded} into $B$ if for all $a, b\in A$ we have $d_A(a,b)=d_B(a,b)$. For a vertex $a\in A$ we write $D_1(a)$ for the set of neighbours of~$a$. Then $|D_1(a)|$ is called the valency of $a$ in $A$.

From now on we fix $n\geq 3$.

\begin{df} A \emph{weak generalized n-gon} $\Gamma$ is a bipartite graph with diameter $n$ and girth $2n$. If $\Gamma$ is \emph{thick}, i.e. if each vertex has valency at least 3, then $\Gamma$ is a \emph{generalized} $n$-gon.\\
A \emph{partial} $n$-gon is a connected bipartite graph of girth at least $2n$.\\
A (partial) $n$-gon $\Gamma_0$ is \emph{non-degenerate}, if $\Gamma_0$ contains  a cycle of length at least $2n+2$ or a path $\g=(x_0,...,x_{n+3})$ with $d_{\Gamma_0}(x_0,x_{n+3})={n+3}$.\\

A (generalized) $n$-gon $\G_0$ contained as a subgraph in a generalized $n$-gon $\G$ is called a (generalized) sub-$n$-gon of $\G$.
\end{df}

\begin{remark}

Note that every thick generalized $n$-gon is non-degenerated.\\

The assumption that a partial $n$-gon is connected is not strictly necessary (and it is not required in \cite{Hall} for $n=3$). Note that for $n=3$ any two distinct points have distance $2$ (and similarly for lines). This is not true anymore for $n>3$, so the requirement that the graph is connected prevents ambiguities.
\end{remark}
 
\begin{df}\label{clean arch}
Let $(x=x_0,\ldots, x_{n-1}=y)$ be a path in $\Gamma$. If  every $x_i, 1\leq i\leq n-2$ has valency $2$ in $\Gamma$, then $(x_1,\ldots, x_{n-2})$ is called a clean arc in $\G$ (with endpoints $x, y$).
A \emph{loose end} is a vertex of valency at most $1$.

A \emph{hat-rack} of length $k\geq n+3$ is a path $(x_0,...,x_k)$  together with subsets of $D_1(x_i)$ for $1\leq i\leq k-1$.
\end{df}
The following definition is due to Tits \cite{Tits}, who first introduced free extensions for generalized polygons, expanding earlier definitions by  M. Hall and Siebenmann \cite{Hall, Siebenmann}.

\begin{df}\label{def: free completion} Let $\Gamma_0$ be a partial $n$-gon. We define the free completion of $\Gamma_0$ by induction on $i<\omega$ as follows:\\
For $i\geq0$ we obtain $\Gamma_{i+1}$ from $\Gamma_i$ by adding a clean arc  between every two elements of $\Gamma_i$ which have distance $n+1$ in $\Gamma_i$. Then $\Gamma=F(\Gamma_0)=\bigcup\limits_{i<\omega}\Gamma_i$ is called the \emph{free n-completion} of $\Gamma_0$. We say that $\Gamma$ is \emph{freely generated} over $\Gamma_0$. 
\end{df}

Note that if $\Gamma_0$ does not contain vertices at distance $\geq n+1$, then $F(\Gamma_0)~=~\Gamma_0$. Also note that by adding a clean arc between vertices  of distance $n+1$ we are creating a new cycle of length $2n$.

\begin{remark}\label{rem:short paths are unique} 
If two elements in a generalized $n$-gon $\Gamma$ have distance less than $n$, there is a unique shortest path in $\Gamma$ connecting them (otherwise we would obtain a short cycle). \\
A weak generalized $n$-gon which contains a $2(n+1)$-cycle is a generalized $n$-gon (conf. \cite{VM}, 1.3).  Hence if $\Gamma_0$ is a partial, non-degenerate $n$-gon, then  $F(\Gamma_0)=\Gamma$ contains a $2(n+1)$-cycle and in fact, $\Gamma$ is an infinite generalized $n$-gon (conf. \cite{VM}) and every vertex $z$ in $F(\Gamma_0)$ has infinite valency.
\end{remark}

We also note the following for future reference:
\begin{remark}\label{rem:opposites}
Let $\G$ be a generalized $n$-gon and let $\g\subset\G$ be a $2n$-cycle. Then for any $x\in\g$ there is a unique $x'\in\g$ with $d(x, x')=n$, ($x'$ is called the opposite of $x$ in $\g$) and for any $y\in D_1(x)\setminus\g$ there is a unique $y'\in D_1(x')$ such that $d(y,y')=n-2$.
Note that the result of adding a clean arc to $\g\cup\{y\}$ is the same as adding a clean arc to $\g\cup\{y'\}$.
\end{remark}

\begin{df}\label{def:generated}
Let $\Gamma$ be a generalized $n$-gon and $A\subset \Gamma$. Then   $\langle A\rangle_\Gamma$ denotes the intersection of all generalized sub-$n$-gons of $\Gamma$ containing $A$. For $\G_0\subset\G$ we put  $\langle A\rangle_{\G_0}= \langle A\rangle_\G\cap\G_0$.
\end{df}
\begin{remark} If $A\subseteq \G_0\subseteq\G$, then $\langle A\rangle_{\G_0}$ is the intersection of all $B\supset A, B\subseteq\G_0,$ such that $A$ is isometrically  contained in $B$. If $A$ is non-degenerate, then $\langle A\rangle_\Gamma$ is a generalized sub-$n$-gon of $\Gamma$, the $n$-gon (not necessarily freely) generated by $A$ in $\Gamma$. Since shortest paths between vertices at distance $n-1$ are unique, clearly $\langle A\rangle_\Gamma\subseteq\acl(A)_\Gamma$.
If $\Gamma=F(A)$, then $\langle A\rangle_\Gamma=\Gamma$.
\end{remark}

We note the following useful observations:
\begin{lem}\label{lem:upward closure}
Let $\Gamma_0$ be a non-degenerate partial n-gon, and let $\Gamma=F(\Gamma_0)~=~\bigcup\Gamma_i$ be as in Definition~\ref{def: free completion}. 
\begin{enumerate}
\item If $A\subset \G_k\setminus \G_i$ is isometrically embedded into $\G_k$, then $\langle A\rangle_\Gamma$ does not intersect $\G_i$  and $\langle A\rangle_\Gamma=F(A)$. 
\item If $A\subset \Gamma_0$ is such that $\Gamma_k\setminus A$ is isometrically embedded into $ \G_k$, then $\langle \G_k\setminus A\rangle_\G$ does not intersect $A$.
\item Any automorphism of $\G_0$ extends to an automorphism of $\G$.
\end{enumerate}
\end{lem}
\begin{proof}
All parts follow directly from the construction:
E.g. for (i) it suffices to show inductively that $\G_1(A)$ is isometrically embedded into $\G_{k+1}\setminus\G_i$. Then (i) follows by induction. Let  $\gamma\subset \G_{k+1}\setminus\G_i$ be   a clean arc connecting $a, b\in A$ with $d_A(a,b)=n+1$.  Any  $c\in\gamma$ has valency $2$, so any path from $c$ to an element in $A$ passes through $a$ or $b$. Since $A$ is isometrically embedded in $\G_k$, the claim follows. The proof for part (ii) is similar and part (iii) is clear.
\end{proof}

Now we can state the main definition of this note, extending the definition of free and open projective planes from \cite{Hall}  to generalized $n$-gons.
\begin{df}\label{open n-gon}\label{df:Gammak}
A (partial) generalized $n$-gon $\Gamma$ is \emph{open} if every finite subgraph contains a loose end or a clean arc.

We call a generalized $n$-gon $\Gamma$ \emph{free} if it is the free $n$-completion of a hat-rack of length at least $n+3$. In particular,
we let $\G^k$ denote the free $n$-completion of the path $\gamma_k=(x_0,\ldots, x_k)$ for $k\geq n+3$.
\end{df}
Note that $\G^k$ is  a free generalized $n$-gon for $k\geq n+3$.

\begin{remark}\label{sopkina}
Clearly, every free generalized $n$-gon is open.
Beware, however,  that the converse is not true in general, see  Proposition~\ref{prop:not free}, but holds for finitely generated generalized $n$-gons, see Proposition~\ref{prop:open+fg is free}.
\end{remark}

Clearly, as observed by \cite{HP} for  the case $n=3$ being an open generalized $n$-gon is a first-order property. We can therefore define:

\begin{df}\label{df:T_n}
Let $T_n$ denote the theory  of  open generalized $n$-gons in the language  of graphs expanded by predicates for the bipartition. 
\end{df}
Note that $T_n$ is $\forall\exists$-axiomatizable.
We start with some easy observations:
\begin{remark}
It follows immediately from Remark~\ref{rem:short paths are unique} and the definition of an open generalized $n$-gon that for $M\models T_n$ and a nondegenerate subgraph $A\subseteq M$ we have $\acl(A)\models T_n$. In other words, every algebraically closed nondegenerate subset of a model of $T_n$ is itself a model of $T_n$.
Clearly, $\acl(A)$ is prime over $A$ (cf. \cite{TZ} 5.3).
\end{remark}

\begin{remark}
Let $T_{n,\gamma}$ be the theory $T_n$ expanded by constants for a path $\gamma=(a_0,\ldots, a_{n+3})$. Then $F(\gamma)$ is the prime model of $T_{n,\gamma}$ since  $F(\gamma)$ is algebraic over $\gamma$, hence countable and atomic, hence prime (cf \cite{TZ}, 4.5.2).
\end{remark}

This is similar to the situation in free groups described in Sela's seminal results, but obviously much easier to prove in the current setting: both theories are strictly stable, and only the 'natural embeddings' are elementary. Namely, we will see later that $\G^k\elemsub \G^m$ if and only if $k\leq m$ and the embedding is the natural one. 

Adapting\footnote{Note that Siebenmann also allows to add vertices of valency $0$.} Siebenmann's  definition for the case $n=3$ \cite{Siebenmann} we define:
\begin{df} If $A$ is a partial $n$-gon, a \emph{hyper-free minimal extension} of $A$ is  an extension by a clean arc between two elements $a, b\in A$ with $d_A(a,b)=n+1$ or by a loose end.\\
 Let $\Gamma$ and $\Gamma'$ be partial $n$-gons. We say that $\Gamma$ is \emph{HF-constructible} from $\Gamma'$ (or over $\Gamma'$) if there is an ordinal $\alpha$ and a sequence $(\Gamma_\beta)_{\beta<\alpha}$ of partial $n$-gons such that
\begin{enumerate}
\item $\Gamma_0=\Gamma';$
\item if $\beta=\gamma+1$, then $\Gamma_\beta$ is a hyper-free minimal extension of $\Gamma_\gamma$;
\item if $\beta$ is a limit ordinal, then $\Gamma_\beta=\bigcup\limits_{\gamma<\beta}\Gamma_\gamma$;
\item $\Gamma=\bigcup\limits_{\beta<\alpha}\Gamma_\beta$.
\end{enumerate}
\end{df}

Clearly, any free completion of a partial $n$-gon $\G_0$ is HF-constructible from $\G_0$.

As in \cite{HP} one can show that any open generalized $n$-gon has an HF-ordering, but since we will not be using this ordering, we omit the details.

\begin{df}\label{df:closed} Let $A, B\subseteq M\models T_n, A\cap B=\emptyset$. We call $B$ \emph{closed over} $A$ if $B$ contains neither a clean arc  with endpoints in $A\cup B$ nor a loose end. We say that $B$ is \emph{open over} $A$ if $B$ contains no finite set closed over $A$ and in this case we write $A\leq_o A\cup B$.
 We write $\hat{A}_M=A\cup \bigcup\{B_0\subset M\ |\ B_0 \textit{ finite and closed over } A\}$.
\end{df}
\begin{remark}\label{rem:union of closed sets is closed}
Note that if $B_1, B_2$ are closed over $A$, then so is $B_1\cup B_2$.
\end{remark}

\begin{lem}\label{lem:finite is constructible}\label{rem:open=HF} If $B$ is open over $A$ and $B\setminus A$ is finite, then $B$ is HF-constructible over $A$. In particular,
 if $A$ is a finite open partial $n$-gon, then $A$ is  HF-constructible from the empty set.
In particular, if $A\leq_o\G$ where $\G$ is a generalized $n$-gon, then $F(A)\cong \langle A\rangle_\G \subseteq\G$.
\end{lem}

\begin{proof}
If $B\setminus A$ is a minimal counterexample, then it cannot contain either a loose end or a clean arch, contradicting the assumption of $B$ being open over $A$. 
\end{proof}

Now consider the class $\calK$ of finite open partial $n$-gons (in the language of bipartite graphs) with strong embeddings given by $\leq_o$.  Note that $\calK$ is contained both in the class of partial $n$-gons considered in \cite{forest}  as well as in the class of partial $n$-gons considered in \cite{veryhomogeneous} (see Lemma~\ref{lem:veryhomogeneous}).

\begin{df}\label{canonical amalgam} For graphs $A\subseteq B, C$, let $B\otimes_A C$ denote the free amalgam of $B$ and $C$ over $A$.\\
Let $A\leq_o C, B$ be open partial $n$-gons (contained in some generalized $n$-gon~$\G$) with $\langle A\rangle_B=\langle A\rangle_C=A$.  Then we call $B\oplus_A C:=F(B\otimes_A C)$ the \emph{canonical amalgam} of $B$ and $C$ over $A$.
\end{df}

The canonical amalgam was used in \cite{forest} (and for $n=3$ in \cite{HP}).
\begin{remark}\label{amalgam HF} (i) If $A\leq_o C, B$ are open partial $n$-gons with $\langle A\rangle_B=\langle A\rangle_C=A$, then  $B,C\leq_o B\otimes_A C\leq_o B\oplus_A C$.
If $B\otimes_A C$ is non-degenerate, then $B\oplus_A C$ is an open generalized $n$-gon.\\ 
(ii) If $B\cap C=A$ and $B\cup C\leq_o\G$ for some generalized $n$-gon $\G$, then $B\cup C\cong B\otimes_A C$ and  $\langle B\cup C\rangle_\G\cong B\oplus_A C$.
\end{remark}

The following is as in \cite{veryhomogeneous, forest} and \cite{HP}:

\begin{prop}\label{calK}
Let $\calK$ be the class of finite connected open partial $n$-gons. Then $(\calK,\leq_o)$
satisfies
\begin{enumerate}
\item (Amalgamation) if $A, B_1, B_2\in\calK$ such that $\iota_i:A\longrightarrow B_i$ and $\iota_i(A)\leq_o B_i, i=1, 2,$ then there exit  $C\in\calK$ and $\kappa_i: B_i\longrightarrow C, i=1, 2$ such that $\kappa_i(B_i)\leq_oC, i=1,2,$ and $\kappa_1(\iota_1(a))=\kappa_2(\iota_2(a))$ for all $a\in A$.
\item (Joint Embedding) for any two graphs $A, B\in\calK$ there is some $C\in\calK$ such that $A, B$ can be strongly embedded (in the sense of $\leq_o$) into $C$.
\end{enumerate}
Hence the limit  $\Gamma_\calK$ exists and is an open generalized $n$-gon. 
\end{prop}
\begin{proof}
Since $\emptyset\in\calK$, it suffices to verify the amalgamation property. Inductively we may assume that $B$ is a minimal hyperfree extension of $A$, so either a clean arc  or a loose end. If $C$ does not contain a copy of $B$ over $A$, then $B\otimes_AC\in\calK$ and this is enough.
\end{proof}

Note that the class $(\calK,\leq_o)$ is unbounded in the sense that for any
$A\in \calK$ there exists some $B\in\calK$ with $A\neq B$ and $A\leq_oB$.

\begin{df}\label{K-sat}
Let $M\models T_n$. Then we say that $M$ is $\calK$-saturated if for all finite sets $A, B\in\calK$ with $A\leq_o B$ and any copy $A'$ of $A$ strongly embedded into $ M$ there is a strong embedding of $B$ over $A'$ into $M$.
\end{df}

Note that by construction, $\G_\calK$ is $\calK$-saturated and that (as in any such Hrushovski construction) every $\calK$-saturated structure is $\calK$-homogeneous in the sense that any partial automorphism between strongly embedded substructures extends to an automorphism.

\begin{thm}\label{complete theory}
For any $n\geq 3$, the theory $T_n$ of open generalized $n$-gons is complete and hence decidable.
\end{thm}
\begin{proof}
Let $M\models T_n$. It suffices to show that $M$ is elementarily equivalent to $\G_\calK$. Clearly we may assume that $M$ is $\omega$-saturated and we claim that any $\omega$-saturated
$M$ is  $\calK$-saturated: let $A\leq_o B$ be from $\calK$ and assume that $A\leq_o M$ (via some strong embedding).  We have to show that we can find  an embedding $B'$ of $B$ into $M$ such that there does not exist a finite set closed over $B'$ in $M$. This is clear if $B$ is an extension of $A$ by a clean arc since such paths are unique. If $B$ is an extension of $A$ by a loose end $b$, then the type of $b$ over $A$ expressing that there is no finite set $D$ closed over $A\cup\{b\}$ is realized in $\G_\calK$, hence it is consistent and therefore realized in $M$ by $\omega$-saturation.
Now both $M$ and  $\G_\calK$ are $\calK$-saturated from which it follows (by standard back-and-forth) that they are partially isomorphic and hence elementarily equivalent.
\end{proof}

We say that a set $B$ neighbours  a set $A$ if every $a\in A$ has a neighbour in $B\setminus A$.

\begin{lem}\label{lem:closed is algebraic}
Let $M\models T_n$, $A\subset M$ finite. Then $M$ does not contain three disjoint sets $B_1, B_2, B_3$ each closed over $A$ and neighbouring $A$. In particular, if $B$ is closed over $A$, then $B$ is algebraic over $A$.
\end{lem}
\begin{proof}
Consider $C=A\cup B_1\cup B_2\cup B_3\subset M$. Then every vertex in $A$ has valency at least~$3$  in $C$ and $C$ contains no clean arc. It follows that $C$ is not open, contradicting $M\models T_n$.

Now suppose $B$ is minimally closed over $A$ and not algebraic over $A$ with $|B\setminus A|$ minimal. Since $B$ is not algebraic over $A$, we find disjoint copies $B_1, B_2, B_3$  of $B$ over $A$, contradicting the first part of the lemma.
\end{proof}
Lemma~\ref{lem:closed is algebraic} directly implies:

\begin{cor}\label{cor:acl implies HF}
If $M\models T_n$ and $A\subseteq M$, then $\acl(A)_M\leq_o M$.
\end{cor}

\begin{df} Let $B\models T_n$ and $A$ a subgraph of $B$. 
We put
 \begin{enumerate}
 \item $\Cl_0(A)_B=A$;
 \item $\Cl_{i+1}(A)_B=\hat{\langle \Cl(A)_i \rangle_B}$ (see Definition~\ref{df:closed});
 \item  $\Cl_B(A)=\bigcup_{i<\omega}\Cl_i(A)_B$.
 \end{enumerate} 
\end{df}

In other words, $\Cl(A)_B$ is the limit  obtained from alternating
 between adding all closed finite subsets, and completing the partial $n$-gons in $B$.
 
\begin{remark} \label{closure HF} For any subset $A$ of $B\models T_n$ we have  and $\Cl_B(A) \leq_o B$ and by Lemma~\ref{lem:closed is algebraic} $\Cl_B(A)\subseteq \acl_B(A)$.
\end{remark}

\begin{thm}\label{thm:acl is elem}
 Let $A, B\models T_n$ and $A\subseteq B$. The following are equivalent:
 \begin{enumerate}
 \item $A=\acl_B(A)$;
 \item $A=\Cl_B(A)$;
 \item $A\leq_o B$;
 \item $A\elemsub B$.
 \end{enumerate}
\end{thm}
\begin{proof}
(i) implies (ii) by Lemma~\ref{lem:closed is algebraic}.\\
(ii) implies (iii) by Remark~\ref{closure HF}.\\
(iii) implies (iv): By taking appropriate elementary extensions we may assume that $A, B$ are $\omega_0$-saturated and hence $\calK$-saturated by the proof of Theorem~\ref{complete theory}. We use Tarski's Test:  Let $B\models\exists x\phi(x,\bar{a})$ for some tuple $\bar{a}\subset A$ and let $b\in B$ such that $B\models\phi(b,\bar{a})$. We find a countable set $A_0$ containing $\bar{a}$ such that $A_0\leq_o A$ and similarly we find a countable set $B_0$ containing $A_0\cup\{b\}$ such that $A_0\leq_o B_0\leq_o B$. Thus by $\calK$-saturation we can embed $B_0$ over $A_0$ into $A$.\\
(iv) implies (i) by Tarski's Test.
\end{proof}

Thus we have

\begin{cor}\label{cor:elem chain} For  $n+3\leq k\leq m\leq\omega$  we have $\G_k\elemsub\G_m$, i.e. the free generalized $n$-gons $\G^k$ form an elementary chain.
\end{cor}

The following lemma will be used in the proof of Theorem~\ref{thm:stable}:

\begin{lem}\label{lem:local char}
Let $M\models T_n$ and $A, C\subseteq M$, $A$ finite and $C$ algebraically closed. Then there exist $a\in A$ and $B_A=\{b_1, b_2\}\subset D_1(a)$ such that for any set $B$ closed over $C\cup A$  and neighbouring  $A$ we have $B\cap B_A\neq \emptyset$.
\end{lem}
\begin{proof}
Suppose otherwise. Then by Remark~\ref{rem:union of closed sets is closed} there is a set $B$ closed over $C\cup A$  and neighbouring  $A$ such that for all $a\in A$ we have $|B\cap D_1(a)|\geq 3$. Since $C$ is algebraically closed,  we know that $B\cup A$ is open over $C$, so contains a loose end or a clean arc which is impossible since all $a\in A$ have valency at least $3$ in $B\cup A$.
\end{proof}

Note that $B_A\subset\acl(AC)$ and $B_A$ might be a singleton.

Exactly as in \cite{pseudospace} and \cite{2-ample} we now define the following notion of independence (see also \cite{HP}).

\begin{df} Let $\M$ be the monster model of $T_n$. For arbitrary subsets $A, B, C $ of $\M$ we say that $B$ and $C$ are independent over $A$, $B{\inds}_A C$, if   $\acl(ABC)=\acl(AB)\oplus_{\acl(A)}\acl(AC)$.
\end{df}

Note that $B{\inds}_A C$ implies in particular $\acl(BA)\cup\acl(AC)\cong \acl(BA)\otimes_A\acl(CA)$.

We will show that $T_n$ is stable by establishing that $\inds$ satisfies the required properties of forking  as in \cite{TZ}, Theorem 8.5.10. where in the notation of that theorem, $B{\inds}_A C$ should be read as $\tp(A/C)\sqsubseteq \tp(A/BC)$.

\begin{thm}\label{thm:stable} The theory $T_n$ of open generalized $n$-gons is stable.\\
In $T_n$, the notion $\inds$ satisfies the properties of stable forking, i.e. 
\begin{itemize}
\item(Invariance)$\inds$ is invariant under $\Aut(\mathfrak{M})$.
\item(Local character) For all $A\subseteq \mathfrak{M}$ finite and $C\subseteq \mathfrak{M}$ arbitrary, there is some countable set $C_0\subseteq C$  such that $A{\inds}_{C_0}C.$
\item(Weak Boundedness) For all $B\subseteq \mathfrak{M}$ finite and $A\subseteq\mathfrak{M}$ arbitrary, there is some cardinal $\mu$, such that there are at most $\mu$ isomorphism types of $B'\subseteq \mathfrak{M}$ over $C$ where $B'\cong_A B$ and $B'{\inds}_A C$.
\item (Existence) For all $B\subseteq \mathfrak{M}$ finite and $A\subseteq C\subseteq \mathfrak{M}$ arbitrary, there is some $B'$ such that $\tp(B/A)=tp(B'/A)$ and $B'{\inds}_A C$.
\item(Transitivity) If $B{\inds}_A C$ and $B{\inds}_{AC}D$ then $B{\inds}_ACD$.
\item(Weak Monotonicity) If $B{\inds}_ACD,$ then $B{\inds}_AC$.
\end{itemize}
\end{thm}

\begin{proof}\ 
\begin{itemize}
\item(Invariance) Clearly ${\inds}$ is invariant under $\Aut(\mathfrak{M})$.
\item(Local Character) Let $A\subset\mathfrak{M}$ be finite  and $C\subseteq\mathfrak{M}$ arbitrary. We 
construct a countable set $C_\infty\subset C$ such that $\acl(A\cup C_\infty)\cup C\leq_o\M$. Then $B=\acl(A\cup C_\infty)$ is countable and by Remark~\ref{amalgam HF} (ii) we have $A{\inds}_{B}C$. 
By Lemma~\ref{lem:local char} there is a finite set $B_A$ which intersects any set $B$ closed over $A\cup C$ and neighbouring $A$. Let $C_A\subset C$ be finite such that $B_A\subset\acl(A\cup C_A)$ and put $C_0=C_A, B_0=\acl(A\cup C_0)$. Suppose inductively that $B_i, C_i$ have been defined, where $B_i, C_i$ are countable. For a finite subset $X\subset B_i$ let $B_X$ be the finite set intersecting any  set  $D$ closed over $X\cup C$ and neighbouring $X$, and let $C_X\subset C$ be finite such that $B_X\subset \acl(C_i\cup C_X)$.
Put $C_{i+1}=C_i\cup\bigcup \{C_X\ |\ X\subset B_i \mathrm{ \ finite}\}$ and $B_{i+1}=\acl(A\cup C_{i+1})$. Note that $C_{i+1}, B_{i+1}$ are again countable. Finally put $C_\infty=\bigcup_{i<\omega} C_i$. 

We now claim that $\acl(A\cup C_\infty)\cup C\leq_o\M$.
Suppose otherwise and let $D$ be a finite set closed over $\acl(A\cup C_\infty)\cup C$ (in particular, by the definition of being closed, $D\cap(\acl(A\cup C_\infty)\cup C)=\emptyset$). Let $Z$ be the set of neighbours of $D$ in $\acl(A\cup C_\infty)$. Since any element of $D$ has at most one neighbour in  $\acl(A\cup C_\infty)$ by Theorem~\ref{thm:acl is elem}, we have  $|Z|\leq |D|$    and hence $Z\subseteq B_i$ for some $i<\omega$. Then by construction $D$ is closed over $Z\cup C\subseteq \acl(A\cup C_\infty)\cup C $ and neighbours $Z$, so intersects the set $B_Z$ nontrivially. But $B_Z\subset B_{i+1}$ by construction. Since $D$ intersects $B_Z$ nontrivially, this contradicts our assumption $D\cap(\acl(A\cup C_\infty)\cup C)=\emptyset$.

\item(Weak Boundedness) 
Let $B\subseteq \mathfrak{M}$ be finite and $A\subseteq C\subseteq \mathfrak{M}$ be arbitrary. If $B\subset\acl(A)$, the claim is obvious. So assume $A, C$ are algebraically closed and $\tp(B_1/A)=\tp(B_2/A)=\tp(B/A)$, so $\acl(B_1A)\cong\acl(B_2A)$ and  $B_1$ and $B_2$ are isometric over $A$.  Hence  from $B_1, B_2{\inds}_A C$, we have 
\[\acl(B_1AC)\cong \acl(B_1A)\oplus_A C\cong \acl(B_2A)\oplus_A C\cong\acl(B_2AC).\]
In particular we have $B_1C\cong B_1\otimes_AC\cong B_2C\cong B_2A\otimes_AC$ and so  $B_1$ and $B_2$ are isometric over $C$. This isometry extends to an isometry from $\acl(B_1AC)$ to $\acl(B_2AC)$ fixing $C$ and since $\acl(B_1AC)$ and $\acl(B_2C)$ are elementary substructures, this extends to an automorphism of $\M$. Hence $\tp(B_1/C)=\tp(B_2/C)$.

\item(Existence) Let  $B\subseteq \mathfrak{M}$ be finite and $A\subset C\subseteq ~\mathfrak{M}$ arbitrary, and consider $D=\acl(BA)\oplus_{\acl(A)}\acl(C)$. We may assume that $C$ is non-degenerate and algebraically closed so that $C\elemsub \M$  and $C\elemsub D$ by Theorem~\ref{thm:acl is elem}. By saturation and homogeneity we can embed $D$ over $C$ into $\M$ in such a way the image of $D$ is an elementary substructure of $\M$. Hence we find $B'$ with $\tp(B/A)=tp(B'/A)$ and $B'{\inds}_A C$.

\item(Transitivity) Let $B{\inds}_A C$ and $B{\inds}_{AC}D,$ so $$\acl(ABC)=\acl(AB)\oplus_{\acl(A)}\acl(AC),$$
 and
 \begin{align*}
 acl(ABCD)&\cong\acl(ABC)\oplus_{\acl(AC)}\acl(ACD)\\
 &\cong \left( \acl(AB)\oplus_{\acl(A)}\acl(AC)\right) \oplus_{\acl(AC)}\acl(ACD)\\
 &\cong \acl(AB)\oplus_{\acl(A)}\acl(ACD),
 \end{align*}
 so $B{\inds}_ACD.$
 \item(Weak Monotonicity) Let $B{\inds}_ACD$, hence 
 \[\G=\acl(ABCD)\cong\acl(AB)\oplus_{\acl(A)}\acl(ACD).\] 
 Now $\acl(AB)\otimes_{\acl(A)}\acl(AC)$ is isometrically embedded into $\acl(AB)\otimes_{\acl(A)}\acl(ACD)$ and hence  by Lemma~\ref{lem:upward closure} we have
 \begin{align*}
 \langle\acl(AB)\otimes_{\acl(A)}\acl(AC)\rangle_\G & = F(\acl(AB)\otimes_{\acl(A)}\acl(AC))\\
 &=\acl(AB)\oplus_{\acl(A)}\acl(AC).
 \end{align*}

\end{itemize}
\end{proof}

As a corollary of the proof we obtain:

\begin{thm} The theory $T$ of open generalized $n$-gons is not superstable.
\end{thm}
\begin{proof}
It suffices to give an example of a finite set $A$ and an algebraically closed set $C$ such that there is no finite set $C_0\subset C$ with $A{\inds}_{C_0}C$.
Let   $\G_0=\gamma_{n+3}=(x_0,\ldots, x_{n+3})\leq_o\M$. Then $\langle\G_0\rangle_\M=\G=\bigcup \G_i$ is the free completion of $\G_0$. For each $0<i<\omega$ let $y_i\in \G_i$ with $d(y_i,\G_{i-1})\geq\frac{n}{2}-1$. Let $z_0\neq x_{n+2}$ be a neighbour of $x_{n+3}$ with   $z_0{\inds}_{x_{n+2}}\G_0$  and let $z_i, 0<i<\omega,$ be  a neighbour of $y_i$ with $z_i{\inds}_{y_i}\G_0z_0\ldots z_{i-1}$. Finally connect $z_i$ and $z_{i-1}$ by a path $\lambda_i$ of length $\geq n-1$ (depending on the parity). Note that the resulting graph  $\tilde{\G}=\G\cup\bigcup_{i<\omega} \lambda_i$ is open with $\G_0\leq_o\tilde{\G}$, and hence we may assume $\tilde{\G}\leq_o\M$.

 Now put $A=\G_0$ and $C=\acl(\{\lambda_i\colon i<\omega\})$. Then by construction there is no finite subset $C_0\subset C$ such that  $A{\inds}_{C_0}C$.
\end{proof}

As in \cite{HP} we can show that independence is not stationary:

\begin{prop}\label{acl neq dcl}
In $T_n$ we have $\acl\neq\dcl$.
\end{prop}
\begin{proof}
Let $\M$ be an $\omega$-saturated model of $T_n$.

{\bf If  $n$ is odd}, let $A=(x_0,\ldots, x_{2n+2}=x_0)\leq_o \M$ be  an ordered  $(2n+2)$-cycle in $\M$. For $i=0,\ldots n$ let $\gamma_i$ be the clean arc from $x_i$ to $x_{i+n+1}$ and let $m_i$ denote the midpoint of $\gamma_i$. Let $C=A\cup\bigcup_{i=0,\ldots n} \gamma_i$. By $\calK$-homogeneity there is an automorphism of $\M$  taking $A$ to the ordered $(2n+2)$-cycle $A'=(x_{n+1},\ldots, x_{2n+2}=x_0\ldots, x_{n+1})$. This will leave the paths $\gamma_i, 0\leq i\leq n,$ invariant and hence fix each $m_i$. This shows that $A\not\subseteq\dcl(m_0,\ldots, m_n)$. On the other hand, $C$ is closed over $\{m_i\ | \ i=0,\ldots n\}$ and hence $A\subset\acl(m_0,\ldots, m_{n-1})$ by Lemma~\ref{lem:closed is algebraic}.

{\bf If $n$ is even}, let $A=(x_0,\ldots, x_{2n}=x_0)$ be an ordered $2n$-cycle and for $i=1,\ldots, n$ let $y_i\notin\{x_{i-1}, x_{i+1}\}$ be a neighbour of $x_i$  and let $z_i$ be the neighbour of $x_{i+n}$ with $d(z_i,y_i)=n-2$.  Let $\gamma_i$ be the (unique) path of length $n-1$ from $y_i$ to $x_{i+n}$ and let $m_i$ be its middle vertex. Note that $z_i\in\Gamma_i$. Let $D=A\cup\{y_1,\ldots, y_n\}$ and assume that $D\leq_o \M$. Put $C=D\cup\bigcup_{i=1,\ldots, n} \gamma_i$.  Then also $D'=A\cup \{z_1,\ldots, z_n\}\leq_o M$. By $\calK$-homogeneity there is an automorphism of $\M$ taking $D$ to $D'$. This automorphism clearly leaves $A$ and $C$ invariant and fixes $m_1,\ldots, m_n$ pointwise, but does not fix any vertex in $A$.  Thus as before we see that $A\not\subseteq\dcl(m_1,\ldots, m_n)$. Since $C$ is closed over $\{ m_i\ |\ i=1\ldots n\}$ we have $A\subseteq C\subseteq\acl(m_1,\ldots, m_n)$ by Lemma~\ref{lem:closed is algebraic}. 
\end{proof}

\section{Elementary substructures}

As noted in \cite{forest} 2.2, if $\Gamma_0$ is isomorphic to $\Delta_0$, their free $n$-completions are also isomorphic. The reverse is obviously not true: In a completion sequence, $\Gamma_1$ and $\Gamma_0$ are not isomorphic, but they clearly have the same free $n$-completion.  \\
There is nevertheless a necessary criterion for the free $n$-completions to be isomorphic. This can be stated in terms of the rank function $\delta_n$ that was introduced in \cite{veryhomogeneous} generalizing the rank function for projective planes introduced by M. Hall~\cite{Hall}. It was  used again in \cite{forest} and \cite{2-ample}.

\begin{df} \label{def:delta}
\begin{enumerate}
\item[(i)] For a finite graph $\Gamma=(V,E)$ with vertex set $V$ and edge set $E$, define  $\delta_n(\Gamma)=(n-1)\cdot|V|-(n-2)\cdot|E|.$
\item[(ii)] A (possibly infinite) graph $\Gamma_0$ is $n$-strong in some graph $\Gamma$, $\Gamma_0\leq_n\Gamma$, if and only if for all finite subgraphs $X$ of $\Gamma$ we have $$\delta_n(X/X\cap\Gamma_0):=\delta_n(X)-\delta_n(X\cap\Gamma_0)\geq 0.$$
\end{enumerate}
\end{df}

\begin{remark} \label{Schnitt} Note that $\delta_n$ is submodular, i.e. if $A\leq_n B$ and $C\subseteq B$, then $A\cap C\leq_nC$.

Let $A$ and $B$ be finite graphs and let $E(A,B)$ denote the edges between elements of $A$ and elements of $B$. Then $$\delta_n(A/B)=\delta_n(A\setminus B)-(n-2)|E(A,B)|.$$
\end{remark}

\begin{remark}[conf. \cite{veryhomogeneous} 2.4]\label{arches n-strong} Let $B$ be a graph which arises from the graph $A$ by successively adding clean arches between elements of distance $n+1$. Then 
 $A\leq_n B$,  $\delta_n(A)=\delta_n(B)$ and hence
if $A\subseteq B\subseteq \Delta$ for some graph $\Delta$ with $A\leq_n\Delta$, then $B\leq_n \Delta$. In particular, if $\G_0$ is a finite partial $n$-gon and $\G=F(\G_0)=\bigcup\G_i$ as in Definition~\ref{def: free completion}, then $\delta_n(\G_i)=\delta_n(\G_0)$ for all $i<\omega$. Hence any finite subset $A_0$ of $\G$ is contained in a finite subset $A\subseteq\G$ with $\delta_n(A)=\delta_n(\G_0)$.
\end{remark}

\begin{lem}[\cite{forest} 2.5] \label{n-strong} Let $\Gamma$ be a generalized $n$-gon which is generated by the graph $\Gamma_0$. The following are equivalent:
\begin{enumerate}
\item[(i)] $\Gamma_0\leq_n\Gamma$
\item[(ii)] $\Gamma=F(\Gamma_0).$
\end{enumerate}
\end{lem}
 
\begin{remark}\label{rem:strong subsets}
Note that for $k\geq n+3$ any finite subset $A_0$ of $\G^k$ is contained in a finite subset $A\subset\G^k$ such that $\delta_n(A)=n-1+k=\delta_n(\g_k)$ and that $n-1+k$ is minimal with that property. Hence, if $A$ and $B$ are finite partial $n$-gons such that $\G(A)\cong\G(B)$, then $\delta_n(A)=\delta_n(B)$. In particular, if $\G^k\cong\G^m$, then $k=m$.
\end{remark}

\begin{df}\label{df:min 0-extension}
If $A\leq_n B$ are finite graphs such that $\delta_n(B/A)=0$ and there is no proper subgraph $A\subset B'\subset B$ with $A\leq_n B'\leq_n B$ then $B$ is called a \emph{minimal} $0$-extension.
\end{df}

\begin{remark}\label{rem:0-extension}
Recall that $\calK$ is the class of finite connected open partial $n$-gons. If $B\in\calK$   is a minimal $0$-extension of $A$, then either $B$ is an extension of $A$ by a clean arc of length $n-2$ or $B$ is closed over $A$ in the sense of Definition~\ref{df:closed}.
\end{remark}

\begin{lem}\label{cl in acl}
Let  $M$ be a model of $T_n$ and $A$ a finite subset of $M$. If $A\subseteq B\subseteq M$ and $\delta_n(B/A)\leq 0$, then $B$ is algebraic over $A$.
\end{lem}
\begin{proof}
If $\delta_n(B/A)< 0$, then $B$ is not HF-constructible over $A$ and hence algebraic over $A$ by Lemma~\ref{lem:closed is algebraic}.
Now suppose  that  $\delta_n(B/A)=0$. By submodularity we can decompose the extension $B$ over $A$ into a finite series $A=B_0\leq_n B_1\ldots\leq_n B_k=B$ where each $B_i$ is a minimal $0$-extension of $B_{i-1}$. Hence it suffices to prove the claim for minimal $0$-extensions and for such extensions the claim follows from Remark~\ref{rem:0-extension} and Lemma~\ref{lem:closed is algebraic}.
\end{proof}

The previous lemma directly implies:
\begin{cor}\label{cor:acl}
Let $\G$ be an open generalized $n$-gon. If $A\subset \G$ is such that every finite set $B_0\supset A$  is contained in a finite set $B$ such that $\delta_n(B)=\delta_n(A)$, then $\G\subseteq \acl(A)$. In particular, any elementary embedding of $\G^k, k\geq n+3,$ into itself is surjective.
\end{cor}

\begin{cor}  [Conf. \cite{HP} 6.3] \label{elmentary chain}  For $n+3\leq k, m\leq \omega$ we have $\Gamma^k\preccurlyeq \Gamma^m$ if and only if $k\leq m$.
\end{cor}

\begin{proof}
The direction from right to left is contained in Corollary~\ref{cor:elem chain}. 
For the direction from left to right suppose $\Gamma^k$ embeds elementarily into $\Gamma^m$ for $m<k$ via $f$, so $f(\G^m)\elemsub f(\G^k)\elemsub\G^m$. By Corollary~\ref{cor:acl} and the direction from right to left, we have $f(\G^m)=\G^m$, contradicting the fact that $\G^m\subsetneq \G^k$.
\end{proof}

To see that $T_n$ has no prime model, we use results from \cite{veryhomogeneous}. Hence we recall the definition of the
 class $\frak{K}$ considered in \cite{veryhomogeneous}. We will show below that $\calK\subseteq\frak{K}$ and hence the results from \cite{veryhomogeneous} apply:
\begin{df}\label{df:veryhomogeneous}
Let $\frak{K}$ be the class of finite partial $n$-gons $A$ such that if $A$ contains a $2k$-cycle for some $k>n$, then $\delta_n(A)\geq 2n+2$.
\end{df}

The following was shown in \cite{veryhomogeneous} Lemma 3.12 (unfortunately the statement there contains a typo):

\begin{lem}\label{lem:veryhomogeneous}
Let $A\in\frak{K}$ with $|A|\geq n+2$. Then $\delta_n(A)\geq 2n$. Moreover, we
have in fact $\delta_n(A) \geq 2n+2$, unless $|A|=n+2$  or $A$ is an ordinary
$n$-gon with either a path with $n-1$ new elements or a loose end attached.
\end{lem}

\begin{prop}\label{prop:delta condition}
If $A\in\calK$  contains a $2k$-cycle for some $k>n$, then $\delta_n(A)\geq 2n+2$. Hence $\calK\subseteq\frak{K}$.
\end{prop}
\begin{proof}
Let $A$ be a minimal counterexample, so $A$  contains a $2k$-cycle for some $k>n$ and $\delta_n(A)<2n+2$. By minimality, $A$ cannot contain a loose end, so $A=A_0\cup\g$ for some clean arc $\g$. Then $\delta_n(A)=\delta_n(A_0)<2n+2$. By minimality $A_0$ does not contain any $2k$-cycle for $k>n$ and hence $A_0\in\frak{K}$. By Lemma~\ref{lem:veryhomogeneous} we know that  $|A_0|=n+2$  or $A_0$ is an ordinary
$n$-gon with either a path with $n-1$ new elements or a loose end attached. But then $A=A_0\cup\g$ does not contain any $2k$-cycle for $k>n$, a contradiction.
\end{proof}

\begin{cor}\label{cor:minimal}
If $\G$ is a generalized $n$ such that every finite set $A_0$ is contained in a finite set $A$ with $\delta(A)=2n+2$, then $\G$ does not contain any proper elementary submodels. In particular,
$\G^{n+3}$ is minimal.
\end{cor}
\begin{proof}
This follows directly from Corollary~\ref{cor:acl} and Proposition~\ref{prop:delta condition}.
\end{proof}

\begin{df}\label{df:Gamma'}
Consider $\G^{n+3}$ and choose copies $(\Gamma_i: i<\omega)$ of $\Gamma^{n+3}$ such that $\Gamma_i\subsetneq \Gamma_{i+1}$. Put $\Gamma'=\bigcup\limits_{i<\omega}\Gamma_i.$ 
\end{df}
Note that $\Gamma'\models T_n$ since $T_n$ is an $\forall\exists$-theory. Also, every finite subset $A_0$ of $\G'$ is contained in a finite set $A\subset\G'$  with $\delta_n(A)=2n+2$. 

\begin{prop}\label{prop:not free}
There exist open generalized $n$-gons which are not free. Specifically, $\G'=\bigcup \G_i$ is not free.
\end{prop}
\begin{proof}
Clearly $\G'$ is not finitely generated as any finite subset is contained in some $\G_i$. So suppose towards a contradiction that $\G'$ is the free completion of an infinite hat-rack. Then for any $k\geq 2n+2$ there exists a subset $X$ of $\G'$ with $\delta_n(X)\geq k$ and $X\leq_n\G'$, a contradiction to the observation that every finite subset $A_0$ of $\G'$ is contained in a finite set $A\subset\G'$  with $\delta_n(A)=2n+2$.  Thus $\G'$ is open and not free.
\end{proof}
\begin{cor}\label{cor:no prime model} The theory $T_n$ of open generalized $n$-gons does not have a prime model. 
\end{cor}
\begin{proof}
By Corollary~\ref{cor:minimal}, $\G^{n+3}$ and $\G'$ (as in Definition~\ref{df:Gamma'}) have no proper elementary substructures. Since they are not isomorphic, this proves the claim.
\end{proof}

Since we can easily find (non-elementary) embeddings of $\G^m$ into $\G^k$ for $m\geq k$ we also obtain:

\begin{cor}\label{cor:no qe} 
The theory of open generalized $n$-gons is not model complete and hence does not have quantifier elimination.
\end{cor}

\begin{remark} Free $\infty$-gons are trees, therefore the theory of free $\infty$-gons is in fact $\omega$-stable as are their higher dimensional generalizations, right-angled buildings and free pseudospaces, see \cite{pseudospace}.
\end{remark}

In \cite{AM}, Ammer also proves
\begin{thm}\label{thm:ample}
The theory $T_n$ has weak elimination of imaginaries and is $1$-ample, but not $2$-ample.
\end{thm}

Furthermore, \cite{AM} extends the proof from \cite{HP} to obtain $2^{\aleph_0}$ many non-isomorphic countable  open generalized $n$-gons for each $n$. Since $T_n$ is not superstable, there are $2^\kappa$ many models of size $\kappa$ for any uncountable~$\kappa$.

\section{Open vs. free}
While we already saw in Proposition~\ref{prop:not free}, that there are open generalized $n$-gons which are not free, 
we show in this final section  that for finitely generated generalized $n$-gons the notions of open and free coincide. For $n=3$ this was proved in \cite{Hall}, Theorem 4.8.

\begin{prop}\label{prop:open+fg is free}
Every finitely generated open generalized $n$-gon is free. 
\end{prop}

For the proof we introduce the following concept:
\begin{df}\label{df:free equivalent}
We call partial $n$-gons $A, B$  free-equivalent if $F(A)\cong F(B)$. 
\end{df}
We begin with the following lemma:

\begin{lem}\label{lem:change order}
Let  $\G=F(A)$ be a generalized $n$-gon and suppose $A$ is constructed from $A_0$ by first attaching a clean arc $\g=(x_1,\ldots, x_{n-2})$ and then attaching loose ends $z_1,\ldots, z_k$ whose respective (unique) neighbours belong to $\g$. Then  there exist $z_1',\ldots, z_k'\in\G\setminus A$ with unique neighbours in $A_0$ such that $A$ is free-equivalent to 
$A_0\cup\{z_1',\ldots, z_k'\}$.
\end{lem}
\begin{proof}
Let $\g'\subset A$ be a $2n$-cycle containing $\g$. Note that the opposites $x_i'$ of $x_i, i=1,\ldots, n-2,$ in $\g$ belong to $A_0$.   By Remark~\ref{rem:opposites} we can replace  $z_i\in D_1(x_j)$ by the appropriate neighbour $z_i'$ of the opposite $x_j'$ of $x_j$ and remove $\g$. 
\end{proof}

\begin{lem}\label{lem:tree to hatrack}
Let  $\G=F(A)$ be a generalized $n$-gon and suppose $A$ does not contain any cycle. Then there is a hat-rack $B$ free-equivalent to $A$.
\end{lem}
\begin{proof}
Let $\g=(x_0,\ldots, x_k)\subset A$ be a simple path (i.e. without repetition of vertices) such that $k\geq n+3$ is maximal. The proof is by induction on the number of vertices of $A$ not incident with $\g$. If $A$ is  a hat-rack, there is nothing to show. So let  $a\in A$ have maximal distance from $\g$. If there is some $x_i\in\g$ such that $d(a,x_i)=n+1$, then let $a'\in\G\setminus A$ be the unique neighbour of $x_i$ with $d(a', a) =n-2$. Let $A'$ be the graph obtained from $A$ be replacing $a$ by $a'$. Then $F(A')=F(A)$. 

If there is no such vertex in $\g$,  let $\g'\subset\G$ be the clean arc connecting $x_0$ and $x_{n+1}$, so $F(A)=F(A\cup\g')$. There is some $y\in\g$ with $d(y,a)=n+1$. Let $y'$ be the neighbour of $y$ with $d(y',a)=n-2$. Then we replace $A$ by $A'=(A\setminus\{a\})\cup \g'\cup\{y'\}$. Thus $F(A)=F(A')$ and the claim follows from Lemma~\ref{lem:change order} and induction.
\end{proof}

We can now give the proof of Proposition~\ref{prop:open+fg is free}:
\begin{proof}
Let $\G$ be an open generalized $n$-gon finitely generated over the finite partial $n$-gon $A$.
We may assume that $A$ is connected. If $\delta_n(A)=k$, then every finite set $A_0\subset\G$ is contained in a finite set $A$ with $\delta(A)\leq k$. Hence we may assume that $A\leq_n\G$ and so $\G\cong F(A)$  by Lemma~\ref{n-strong}. Therefore it suffices to show that there is a finite hat-rack $B$ free-equivalent to $A$. 

By Lemma~\ref{lem:finite is constructible} consider a construction of $A$ over the emptyset. Clearly we may assume that the last construction step is the addition of a loose end. We now do induction over the number of steps adding a clean arch. If this number is zero, then $A$ contains no cycles  and the claim follows from Lemma~\ref{lem:tree to hatrack}. Now suppose $A$ is obtained from $A_0$ by adding  a clean arc $\g$ and then adding a number of loose ends $z_1,\ldots, z_k$ (where the loose ends may be attached consecutively at a previous loose end). If all loose ends are incident with $\g$, then we finish using Lemma~\ref{lem:change order}. Otherwise, we inductively reduce the distance of the loose ends by replacing them by a loose end at smaller distance to $A$: if $z_i$ is a loose end, there is some $x\in A$ with $d(z_i,x)=n+1$ and such that $x$ is not a loose end in $A$. Now replace $z_i$ by the unique $z_i'\in D_1(x)$ with $d(z_i,z_i')=n-2$.  In this way we reduce to the case in Lemma~\ref{lem:change order} and finish.
\end{proof}

\begin{remark}\label{rem:delta char}
Using similar arguments one can also show that
for finitely generated $\G, \G'\models T_n$ we have $\G\cong\G'$ if and only $\G=F(A), \G'=F(B)$ for finite $A, B$ such that $\delta_n(A)=\delta_n(B)$. We leave the proof as an excercise for the interested reader.
\end{remark}

\end{document}